\documentclass[12pt]{amsart}
\usepackage{amsmath,amsthm,amsfonts,amssymb,mathrsfs}
\date{\today}

\usepackage{color}

\usepackage{hyperref}

\newtheorem{theorem}{Theorem}[section]

\newtheorem{proposition}[theorem]{Proposition}
\newtheorem{corollary}[theorem]{Corollary}
\newtheorem{lemma}[theorem]{Lemma}
\theoremstyle{definition}
\newtheorem{example}[theorem]{Example}%[section]
\newtheorem{remark}[theorem]{Remark}%[section]

\begin{document}

\title[On the dichotomy of a locally compact semitopological bicyclic monoid ...] {On the dichotomy of a locally compact semitopological bicyclic monoid with adjoined zero}

\author{Oleg~Gutik}
\address{Faculty of Mathematics, National University of Lviv,
Universytetska 1, Lviv, 79000, Ukraine}
\email{o\underline{\hskip5pt}\,gutik@franko.lviv.ua,
ovgutik@yahoo.com}



\begin{abstract}
We prove that a Hausdorff locally compact semitopological bicyclic semigroup with adjoined zero $\mathscr{C}^0$ is either compact or discrete. Also we show that the similar statement holds for a locally compact semitopological bicyclic semigroup with an adjoined compact ideal and construct an example which witnesses that a counterpart of the statements does not hold when $\mathscr{C}^0$ is a \v{C}ech-complete metrizable topological inverse semigroup.
\end{abstract}

\maketitle

\section{Introduction and preliminaries}

Further we shall follow the terminology of \cite{Carruth-Hildebrant-Koch-1983-1986, Clifford-Preston-1961-1967, Engelking-1989, Ruppert-1984}.
Given a semigroup $S$, we shall
denote the set of idempotents of $S$ by $E(S)$.
A semigroup $S$ with the adjoined  zero  will be denoted by
$S^0$  (cf.
\cite{Clifford-Preston-1961-1967}).

A semigroup $S$ is called \emph{inverse} if for every $x\in S$ there exists a unique $y\in S$ such that $xyx=x$ and $yxy=y$. Later such an element $y$ will be denoted by $x^{-1}$ and will be called the \emph{inverse of} $x$. A map $\operatorname{\textsf{inv}}\colon S\to S$ which assigns to every $s\in S$ its inverse is called \emph{inversion}.

In this paper all topological spaces are Hausdorff. If $Y$ is a subspace of a topological space $X$ and $A\subseteq Y$, then by $\operatorname{cl}_Y(A)$ we denote the topological closure of $A$ in $Y$\textcolor[rgb]{1.00,0.00,0.00}.

A \emph{semitopological} (\emph{topological}) \emph{semigroup} is a topological space with separately continuous (jointly continuous) semigroup operations. An inverse topological semigroup with continuous inversion is called a \emph{topological inverse semigroup}.

We recall that a topological space $X$ is:
\begin{itemize}
  \item \emph{locally compact} if every point $x$ of $X$ has an open neighbourhood $U(x)$ with the compact closure $\operatorname{cl}_X(U(x))$;
  \item \emph{\v{C}ech-complete} if $X$ is Tychonoff and there exists a compactification $cX$ of $X$ such that the remainder $cX\setminus c(X)$ is an $F_\sigma$-set in $cX$.
\end{itemize}

The \emph{bicyclic semigroup} (or the \emph{bicyclic monoid}) ${\mathscr{C}}(p,q)$ is a semigroup with the identity $1$ generated by two elements $p$ and $q$ with only one condition $pq=1$. The distinct elements of the bicyclic monoid are exhibited in the following array:
\begin{equation*}
\begin{array}{ccccc}
  1      & p      & p^2    & p^3     & \cdots \\
  q      & qp     & qp^2   & qp^3    & \cdots \\
  q^2    & q^2p   & q^2p^2 & q^2p^3  & \cdots \\
  q^3    & q^3p   & q^3p^2 & q^3p^3  & \cdots \\
  \vdots & \vdots & \vdots & \vdots  & \ddots
\end{array}
\end{equation*}

By $\mathscr{C}^0$ we denote the bicyclic monoid with adjoined zero, i.e., $\mathscr{C}^0=\mathscr{C}(p,q)\sqcup\{0\}$.

The bicyclic monoid is a combinatorial bisimple $F$-inverse semigroup and it plays an
important role in the algebraic theory of semigroups and in the
theory of topological semigroups. For example the well-known
Andersen's result~\cite{Andersen-1952} states that a ($0$--)simple
semigroup with an idempotent is completely ($0$--)simple if and only if it does not
contain an isomorphic copy of the bicyclic semigroup. The bicyclic semigroup admits only the
discrete semigroup topology and if a topological semigroup $S$ contains it as a dense subsemigroup then
${\mathscr{C}}(p,q)$ is an open subset of $S$~\cite{Eberhart-Selden-1969}. Bertman and  West in \cite{Bertman-West-1976} extended this result for the case of semitopological semigroups. Stable and $\Gamma$-compact topological semigroups do not contain the
bicyclic semigroup~\cite{Anderson-Hunter-Koch-1965, Hildebrant-Koch-1988}. The problem of an embedding of the bicyclic monoid into compact-like topological semigroups is discussed in \cite{Banakh-Dimitrova-Gutik-2009, Banakh-Dimitrova-Gutik-2010, Gutik-Repovs-2007}.

In \cite{Eberhart-Selden-1969} Eberhart and Selden proved that if the bicyclic monoid ${\mathscr{C}}(p,q)$ is a dense subsemigroup of a topological monoid $S$ and $I=S\setminus{\mathscr{C}}(p,q)\neq\varnothing$ then $I$ is a two-sided ideal of the semigroup $S$. Also, there they described the closure of the bicyclic monoid ${\mathscr{C}}(p,q)$ in a locally compact topological inverse semigroup. The closure of the bicyclic monoid in a countably compact (pseudocompact) topological semigroups was studied in~\cite{Banakh-Dimitrova-Gutik-2010}.

The well known A.~Weil Theorem states that \emph{every locally compact monothetic topological group $G$} (i.e., $G$ contains a cyclic dense subgroup) \emph{is either compact or discrete} (see \cite{Weil-1938}). Locally compact and compact monothetic topological semigroups was studied by Hewitt \cite{Hewitt-1956},  Hofmann \cite{Hofmann-1960}, Koch  \cite{Koch-1957}, Numakura \cite{Numakura-1952} and others (see more information on this topics in the books \cite{Carruth-Hildebrant-Koch-1983-1986} and \cite{Hofmann-Mostert-1966}). Koch in \cite{Koch-1969} posed the following problem: ``\emph{If $S$ is a locally compact monothetic semigroup and $S$ has an identity, must $S$ be compact?}'' (see \cite[Vol. 2, p.~144]{Carruth-Hildebrant-Koch-1983-1986}). From the other side, Zelenyuk in \cite{Zelenyuk-1988} constructed a countable locally compact topological semigroup without unit which is neither compact nor discrete.

In this paper we prove that a Hausdorff locally compact semitopological bicyclic semigroup with adjoined zero $\mathscr{C}^0$ is either compact or discrete. Also we show that the similar statement holds for a locally compact semitopological bicyclic semigroup with an adjoined compact ideal and construct an example which witnesses that a counterpart of the statements does not hold when $\mathscr{C}^0$ is a \v{C}ech-complete metrizable topological inverse semigroup.

\section{On a locally compact semitopological bicyclic semigroup with adjoined zero}

The following proposition generalizes Theorem~I.3 from \cite{Eberhart-Selden-1969}.

\begin{proposition}\label{proposition-2.1}
If the bicyclic monoid ${\mathscr{C}}(p,q)$ is a dense subsemigroup of a semitopological monoid $S$ and $I=S\setminus{\mathscr{C}}(p,q)\neq\varnothing$ then $I$ is a two-sided ideal of the semigroup $S$.
\end{proposition}

\begin{proof}
Fix an arbitrary element $y\in I$. If $xy=z\notin I$ for some $x\in{\mathscr{C}}(p,q)$ then there exists an open neighbourhood $U(y)$ of the point $y$ in the space $S$ such that $\{x\}\cdot U(y)=\{z\}\subset{\mathscr{C}}(p,q)$. The neighbourhood $U(y)$ contains infinitely many elements of the semigroup ${\mathscr{C}}(p,q)$. This contradicts Lemma~I.1~\cite{Eberhart-Selden-1969}, which states that for each $v,w\in \mathscr{C}(p,q)$ both sets $\{u\in \mathscr{C}(p,q) \colon vu=w\}$ and $\{u\in \mathscr{C}(p,q) \colon uv=w\}$ are finite. The obtained contradiction implies that $xy\in I$ for all $x\in  \mathscr{C}(p,q)$ and $y\in I$. The proof of the statement that $yx\in I$ for all $x\in  \mathscr{C}(p,q)$ and $y\in I$ is similar.

Suppose to the contrary that $xy=w\notin I$ for some $x,y\in I$. Then $w\in {\mathscr{C}}(p,q)$ and the separate continuity of the semigroup operation in $S$ implies that there exist open neighbourhoods $U(x)$ and $U(y)$ of the points $x$ and $y$ in $S$, respectively, such that $\{x\}\cdot U(y)=\{w\}$ and $U(x)\cdot \{y\}=\{w\}$. Since both neighbourhoods $U(x)$ and $U(y)$ contain infinitely many elements of the semigroup ${\mathscr{C}}(p,q)$, both equalities $\{x\}\cdot U(y)=\{w\}$ and $U(x)\cdot \{y\}=\{w\}$ contradict mentioned above Lemma~I.1 from~\cite{Eberhart-Selden-1969}. The obtained contradiction implies that $xy\in I$.
\end{proof}

For every non-negative integer $n$ we put
\begin{equation*}
    \mathscr{C}[q^n]=\left\{q^np^i\in\mathscr{C}(p,q)\colon i=0,1,2,\ldots\right\} \: \hbox{ and } \:  \mathscr{C}[p^n]=\left\{q^ip^n\in\mathscr{C}(p,q)\colon i=0,1,2,\ldots\right\}.
\end{equation*}

\begin{lemma}\label{lemma-2.2}
Let $(\mathscr{C}^0,\tau)$ be a locally compact semitopological semigroup. Then the following assertions hold:
\begin{itemize}
  \item[$(1)$] for every open neighbourhood $U(0)$ of zero in $(\mathscr{C}^0,\tau)$ there exists an open compact neighbourhood $V(0)$ of zero in $(\mathscr{C}^0,\tau)$ such that $V(0)\subseteq U(0)$;
  \item[$(2)$] for every open compact neighbourhood $U(0)$ of zero in $(\mathscr{C}^0,\tau)$ and every open neighbourhood $V(0)$ of zero in $(\mathscr{C}^0,\tau)$ the set $U(0)\cap V(0)$ is compact and open, and the set $U(0)\setminus V(0)$ is finite.
\end{itemize}
\end{lemma}

\begin{proof}
The statements of the lemma are trivial in the case when $\tau$ is the discrete topology on $\mathscr{C}^0$, and hence later we shall assume that the topology $\tau$ is non-discrete.

$(1)$ Let $U(0)$ be an arbitrary open neighbourhood of zero in $(\mathscr{C}^0,\tau)$. By Theorem~3.3.1 from \cite{Engelking-1989} the space $(\mathscr{C}^0,\tau)$ is regular. Since it is locally compact, there exists an open neighbourhood $V(0)\subseteq U(0)$ of zero in $(\mathscr{C}^0,\tau)$ such that $\operatorname{cl}_{\mathscr{C}^0}(V(0))\subseteq U(0)$. Since all non-zero elements of the semigroup $\mathscr{C}^0$ are isolated points in $(\mathscr{C}^0,\tau)$, $\operatorname{cl}_{\mathscr{C}^0}(V(0))=V(0)$, and hence our assertion holds.

$(2)$ Let $U(0)$ be an arbitrary compact open neighbourhood of zero in $(\mathscr{C}^0,\tau)$. Then for an arbitrary open neighbourhood $V(0)$ of zero in $(\mathscr{C}^0,\tau)$ the family
\begin{equation*}
\mathscr{U}=\left\{V(0),\left\{\{x\}\colon x\in U(0)\setminus V(0)\right\}\right\}
\end{equation*}
is an open cover of $U(0)$. Since the family $\mathscr{U}$ is disjoint, it is finite. So the set $U(0)\setminus V(0)$ is finite and the set $U(0)\cap V(0)$ is compact.
\end{proof}

\begin{lemma}\label{lemma-2.3}
If $(\mathscr{C}^0,\tau)$ is a locally compact non-discrete semitopological semigroup, then for each open neighbourhood $U(0)$ of zero in $(\mathscr{C}^0,\tau)$ there exist non-negative integers $i$ and $j$ such that both sets $\mathscr{C}[q^i]\cap U(0)$ and $\mathscr{C}[p^j]\cap U(0)$ are infinite.
\end{lemma}

\begin{proof}
By Lemma~\ref{lemma-2.2}$(1)$, without loss of generality we may assume that $U(0)$ is a compact open neighbourhood of zero $0$ in $(\mathscr{C}^0,\tau)$. Put
\begin{equation*}
    V_q(0)=\left\{x\in U(0)\colon x\cdot q\in U(0)\right\} \qquad \hbox{and} \qquad V_p(0)=\left\{x\in U(0)\colon p\cdot x \in U(0)\right\}.
\end{equation*}
If the set $\mathscr{C}[q^i]\cap U(0)$ is finite for any non-negative integer $i$, then the formula
\begin{equation}\label{eq-2.1}
    q^ip^l\cdot q=
    \left\{
      \begin{array}{ll}
        q^{i+1},    & \hbox{if~} l=0;\\
        q^ip^{l-1}, & \hbox{if~} l \hbox{~is a positive integer},
      \end{array}
    \right.
\end{equation}
implies that the right translation $\rho_q\colon \mathscr{C}^0\rightarrow\mathscr{C}^0\colon x\mapsto x\cdot q$ shifts all non-zero elements of the neighbourhood $V_q(0)$. Then $U(0)\setminus V_q(0)$ is an infinite subset of ${\mathscr{C}}(p,q)$, which contradicts Lemma~\ref{lemma-2.2}$(2)$. Similarly, if the set $\mathscr{C}[p^j]\cap U(0)$ is finite for any non-negative integer $j$, then the formula
\begin{equation}\label{eq-2.2}
   p\cdot q^jp^l=
    \left\{
      \begin{array}{ll}
        p^{l+1},      & \hbox{if~} j=0;\\
        q^{j-1}p^{l}, & \hbox{if~} j \hbox{~is a positive integer},
      \end{array}
    \right.
\end{equation}
implies that the left translation $\lambda_p\colon \mathscr{C}^0\rightarrow\mathscr{C}^0\colon x\mapsto p\cdot x$ shifts all non-zero elements of the neighbourhood $V_p(0)$. This implies that $U(0)\setminus V_p(0)$ is an infinite subset of ${\mathscr{C}}(p,q)$, which contradicts Lemma~\ref{lemma-2.2}$(2)$.
\end{proof}

\begin{lemma}\label{lemma-2.4}
Let $(\mathscr{C}^0,\tau)$ be a locally compact non-discrete semitopological semigroup. Then there exist non-negative integers $i$ and $j$ such that $\mathscr{C}[q^i]\setminus U(0)$ and $\mathscr{C}[p^j]\setminus U(0)$ are finite for every open neighbourhood $U(0)$ of zero $0$ in $(\mathscr{C}^0,\tau)$.
\end{lemma}

\begin{proof}
Fix an arbitrary open compact neighbourhood $U_0(0)$ of zero in $(\mathscr{C}^0,\tau)$. Then Lemma~\ref{lemma-2.3} implies that there exist non-negative integers $i$ and $j$ such that both sets $\mathscr{C}[q^i]\cap U_0(0)$ and $\mathscr{C}[p^j]\cap U_0(0)$ are infinite. Let $U(0)$ be an arbitrary open neighbourhood of zero in $(\mathscr{C}^0,\tau)$. By Lemma~\ref{lemma-2.2}$(2)$, the set $U_0(0)\setminus U(0)$ is finite. By Lemma~\ref{lemma-2.2}$(1)$, there exists an open compact neighbourhood $U'(0)\subseteq U(0)$ of zero in $(\mathscr{C}^0,\tau)$.

Now, Lemma~\ref{lemma-2.2}$(1)$ and the separate continuity of the semigroup operation in $(\mathscr{C}^0,\tau)$ imply that there exists an open compact neighbourhood $V(0)$ of zero $0$ in $(\mathscr{C}^0,\tau)$ such that
\begin{equation*}
    V(0)\subseteq U'(0), \qquad V(0)\cdot q\subseteq U'(0)\qquad \hbox{and} \qquad p\cdot V(0)\subseteq U'(0).
\end{equation*}
If the set $\mathscr{C}[q^i]\setminus U(0)$ is infinite, then formula (\ref{eq-2.1}) implies that the right translation $\rho_q\colon \mathscr{C}^0\rightarrow\mathscr{C}^0\colon x\mapsto x\cdot q$ shifts all non-zero elements of the neighbourhood $V(0)$ and hence the inclusion $V(0)\cdot q\subseteq U'(0)$ implies that $U'(0)\setminus V(0)$ is an infinite set, which contradicts Lemma~\ref{lemma-2.2}$(2)$. Hence the set $\mathscr{C}[q^i]\setminus U(0)$ is finite. Similarly, if the set $\mathscr{C}[p^j]\setminus U(0)$ is infinite, then by formula (\ref{eq-2.2}) we have that the left translation $\lambda_p\colon \mathscr{C}^0\rightarrow\mathscr{C}^0\colon x\mapsto p\cdot x$ shifts all non-zero elements of the neighbourhood $V(0)$ and hence the by inclusion $p\cdot V(0)\subseteq U'(0)$ we obtain that $U'(0)\setminus V(0)$ is an infinite set, which contradicts Lemma~\ref{lemma-2.2}$(2)$. Therefore, the set $\mathscr{C}[p^j]\setminus U(0)$ is finite as well.
\end{proof}

\begin{lemma}\label{lemma-2.5}
Let $(\mathscr{C}^0,\tau)$ be a locally compact non-discrete semitopological semigroup. Then for every open neighbourhood $U(0)$ of zero $0$ in $(\mathscr{C}^0,\tau)$ and any non-negative integer $i$ both sets $\mathscr{C}[q^i]\setminus U(0)$ and $\mathscr{C}[p^i]\setminus U(0)$ are finite.
\end{lemma}

\begin{proof}
By Lemma~\ref{lemma-2.2}$(1)$, without loss of generality we may assume that the open neighbourhood $U(0)$ is compact.
By Lemma~\ref{lemma-2.4} there exists a non-negative integer $i_0$ such that $\mathscr{C}[q^{i_0}]\setminus U'(0)$ is finite for any open compact neighbourhood $U'(0)$ of zero $0$ in $(\mathscr{C}^0,\tau)$.

Fix an arbitrary non-negative integer $i\neq i_0$.
If $i<i_0$, then the separate continuity of the semigroup operation in $(\mathscr{C}^0,\tau)$ implies that there exists an open compact neighbourhood $V(0)\subseteq U(0)$ of zero $0$ in $(\mathscr{C}^0,\tau)$ such that
$p^{i_0-i}\cdot V(0)\subseteq U(0)$. Then
\begin{equation}\label{eq-2.3}
    p^{i_0-i}\cdot q^{i_0}p^l=q^{i}p^l, \qquad \hbox{for any non-negative integer~} l.
\end{equation}
The set $\mathscr{C}[q^{i_0}]\setminus V(0)$ is finite, and hence by (\ref{eq-2.3}) the set $\mathscr{C}[q^{i}]\setminus U(0)\subseteq \mathscr{C}[q^{i}]\setminus \left(p^{i_0-i}\cdot V(0)\right)$ is finite as well.

If $i>i_0$, then the separate continuity of the semigroup operation in $(\mathscr{C}^0,\tau)$ implies that there exists an open compact neighbourhood $W(0)\subseteq U(0)$ of zero $0$ in $(\mathscr{C}^0,\tau)$ such that
$q^{i-i_0}\cdot W(0)\subseteq U(0)$. Then
\begin{equation}\label{eq-2.4}
    q^{i-i_0}\cdot q^{i_0}p^l=q^{i}p^l, \qquad \hbox{for any non-negative integer~} l,
\end{equation}
The set $\mathscr{C}[q^{i_0}]\setminus W(0)$ is finite, and hence (\ref{eq-2.4}) implies that the set $\mathscr{C}[q^{i}]\setminus U(0)\subseteq \mathscr{C}[q^{i}]\setminus \left(q^{i-i_0}\cdot W(0)\right)$ is finite as well.

The proof of finiteness of the set $\mathscr{C}[p^i]\setminus U(0)$ is similar.
\end{proof}

\begin{lemma}\label{lemma-2.6}
Let $(\mathscr{C}^0,\tau)$ be a non-discrete locally compact semitopological semigroup. Then for every open neighbourhood $U(0)$ of zero $0$ in $(\mathscr{C}^0,\tau)$ the set $\mathscr{C}^0\setminus U(0)$ is finite.
\end{lemma}

\begin{proof}
Suppose to the contrary that there exists an open neighbourhood $U(0)$ of zero $0$ in $(\mathscr{C}^0,\tau)$ such that $\mathscr{C}^0\setminus U(0)$ is infinite. Lemma~\ref{lemma-2.2}$(1)$ implies that without loss of generality we may assume that the neighbourhood $U(0)$ is compact.

Now, the separate continuity of the semigroup operation in $(\mathscr{C}^0,\tau)$ implies that there exists an open neighbourhood $V(0)\subseteq U(0)$ of zero $0$ in $(\mathscr{C}^0,\tau)$ such that $p\cdot V(0)\subseteq U(0)$. By Lemma~\ref{lemma-2.5} for every non-negative integer $n$ both sets $\mathscr{C}[q^n]\setminus U(0)$ and $\mathscr{C}[p^n]\setminus U(0)$ are finite. Thus, the following conditions hold:
\begin{itemize}
  \item[$(i)$] $U(0)\cup\bigcup_{n=0}^m\left(\mathscr{C}[q^n]\cup \mathscr{C}[p^n]\right)\neq\mathscr{C}^0$ for every positive integer $m$;
  \item[$(ii)$] for every positive integer $k$ there exists a non-negative integer $k_{\texttt{max}}$ such that $\left\{q^kp^j\colon j\geqslant k_{\texttt{max}}\right\}\subset U(0)$.
\end{itemize}
We have $p\cdot q^kp^l=q^{k-1}p^k$ for any integers $k\geqslant 1$ and $l$. This and conditions $(i)$ and $(ii)$ imply that the set $U(0)\setminus V(0)$ is infinite, which contradicts Lemma~\ref{lemma-2.2}$(2)$. The obtained contradiction implies the statement of the lemma.
\end{proof}

The following simple example shows that on the semigroup $\mathscr{C}^0$ there exists a topology $\tau_{\operatorname{\textsf{Ac}}}$ such that $(\mathscr{C}^0,\tau_{\operatorname{\textsf{Ac}}})$ is a compact semitopological semigroup.

\begin{example}\label{example-2.7}
On the semigroup $\mathscr{C}^0$ we define a topology $\tau_{\operatorname{\textsf{Ac}}}$ in the following way:
\begin{itemize}
  \item[$(i)$] every element of the bicyclic monoid ${\mathscr{C}}(p,q)$ is an isolated point in the space $(\mathscr{C}^0,\tau_{\operatorname{\textsf{Ac}}})$;
  \item[$(ii)$] the family $\mathscr{B}(0)=\left\{U\subseteq \mathscr{C}^0\colon U\ni 0 \hbox{~and~} {\mathscr{C}}(p,q)\setminus U \hbox{~is finite}\right\}$ determines a base of the topology $\tau_{\operatorname{\textsf{Ac}}}$ at zero $0\in\mathscr{C}^0$,
\end{itemize}
i.e., $\tau_{\operatorname{\textsf{Ac}}}$ is the topology of the Alexandroff one-point compactification of the discrete space ${\mathscr{C}}(p,q)$ with the remainder $\{0\}$. The semigroup operation in $(\mathscr{C}^0,\tau_{\operatorname{\textsf{Ac}}})$ is separately continuous, because all elements of the bicyclic semigroup ${\mathscr{C}}(p,q)$ are isolated points in the space $(\mathscr{C}^0,\tau_{\operatorname{\textsf{Ac}}})$.
\end{example}

\begin{remark}\label{remark-2.8}
In \cite{Bertman-West-1976} Bertman and  West showed that the discrete topology $\tau_{\textsf{d}}$ is a unique topology on the bicyclic monoid ${\mathscr{C}}(p,q)$ such that ${\mathscr{C}}(p,q)$ is a semitopological semigroup. So $\tau_{\operatorname{\textsf{Ac}}}$ is the unique compact topology on $\mathscr{C}^0$ such that $(\mathscr{C}^0,\tau_{\operatorname{\textsf{Ac}}})$ is a compact semitopological semigroup.
\end{remark}

Lemma~\ref{lemma-2.6} and Remark~\ref{remark-2.8} imply the following dichotomy for a locally compact semitopological semigroup $\mathscr{C}^0$.

\begin{theorem}\label{theorem-2.9}
If $\mathscr{C}^0$ is a Hausdorff locally compact semitopological semigroup, then either $\mathscr{C}^0$ is discrete or
$\mathscr{C}^0$ is topologically isomorphic to $(\mathscr{C}^0,\tau_{\operatorname{\textsf{Ac}}})$.
\end{theorem}

Since the bicyclic monoid $\mathscr{C}(p,q)$ does not embeds into any Hausdorff compact topological semigroup \cite{Anderson-Hunter-Koch-1965}, Theorem~\ref{theorem-2.9} implies the following corollary.

\begin{corollary}\label{corollary-2.10}
If $\mathscr{C}^0$ is a Hausdorff locally compact topological semigroup, then $\mathscr{C}^0$ is discrete.
\end{corollary}

The following example shows that a counterpart of the statement of Corollary~\ref{corollary-2.10} does not hold when $\mathscr{C}^0$ is a \v{C}ech-complete metrizable topological inverse semigroup.

\begin{example}\label{example-2.11}
On the semigroup $\mathscr{C}^0$ we define a topology $\tau_{\operatorname{\textsf{1}}}$ in the following way:
\begin{itemize}
  \item[$(i)$] every element of the bicyclic monoid ${\mathscr{C}}(p,q)$ is an isolated point in the space $(\mathscr{C}^0,\tau_{\operatorname{\textsf{1}}})$;
  \item[$(ii)$] the family $\mathscr{B}(0)=\left\{U_n \colon n=0,1,2,3,\ldots\right\}$, where
   \begin{equation*}
   U_n=\{0\}\cup\left\{q^ip^j\in\mathscr{C}(p,q)\colon i,j>n\right\},
\end{equation*}
determines a base of the topology $\tau_{\operatorname{\textsf{1}}}$ at zero $0\in\mathscr{C}^0$.
\end{itemize}
It is obvious that $(\mathscr{C}^0,\tau_{\operatorname{\textsf{1}}})$ is first countable space and the arguments presented in \cite[p.~68]{Gutik-1996} show that $(\mathscr{C}^0,\tau_{\operatorname{\textsf{1}}})$ is a Hausdorff topological inverse semigroup.

First we observe that each element of the family $\mathscr{B}(0)$ is an open closed subset of $(\mathscr{C}^0,\tau_{\operatorname{\textsf{1}}})$, and hence the space $(\mathscr{C}^0,\tau_{\operatorname{\textsf{1}}})$ is regular. Since the set $\mathscr{C}^0$ is countable, the definition of the topology $\tau_{\operatorname{\textsf{1}}}$ implies that $(\mathscr{C}^0,\tau_{\operatorname{\textsf{1}}})$ is second countable, and hence by Theorem~4.2.9 from \cite{Engelking-1989} the space $(\mathscr{C}^0,\tau_{\operatorname{\textsf{1}}})$ is metrizable. Also, it is obvious that the space $(\mathscr{C}^0,\tau_{\operatorname{\textsf{1}}})$ is \v{C}ech-complete, as a union two \v{C}ech-complete spaces: that are the discrete space $\mathscr{C}(p,q)$ and the singleton space $\{0\}$.
\end{example}

%%%%%%%%%%%%%%%%%%%%%%%%%%%%%%%%%%%%%%%%%%%%%%%%%%%%%%%%%%%%

\section{On a locally compact semitopological bicyclic semigroup with an adjoined compact ideal}

Later we need the following notions. A continuous map $f\colon X\to Y$ from a topological space $X$ into a topological space $Y$ is called:
\begin{itemize}
  \item[$\bullet$] \emph{quotient} if the set $f^{-1}(U)$ is open in $X$ if and only if $U$ is open in $Y$ (see \cite{Moore-1925} and \cite[Section~2.4]{Engelking-1989});
  \item[$\bullet$] \emph{hereditarily quotient} or \emph{pseudoopen} if for every $B\subset Y$ the restriction $f|_{B}\colon f^{-1}(B)\\ \rightarrow B$ of $f$ is a quotient map (see \cite{McDougle-1958, McDougle-1959, Arkhangelskii-1963} and \cite[Section~2.4]{Engelking-1989});
  \item[$\bullet$] \emph{closed} if $f(F)$ is closed in $Y$ for every closed subset $F$ in $X$;
  \item[$\bullet$] \emph{perfect} if $X$ is Hausdorff, $f$ is a closed map and all fibers $f^{-1}(y)$ are compact subsets of $X$ \cite{Vainstein-1947}.
\end{itemize}
Every closed map and every hereditarily quotient map are quotient \cite{Engelking-1989}. Moreover, a continuous map $f\colon X\to Y$ from a topological space $X$ onto a topological space $Y$ is hereditarily quotient if and only if for every $y\in Y$ and every open subset $U$ in $X$ which contains $f^{-1}(y)$ we have that $y\in\operatorname{int}_Y(f(U))$ (see \cite[2.4.F]{Engelking-1989}).

Later we need the following trivial lemma, which follows from separate continuity of the semigroup operation in semitopological semigroups.

\begin{lemma}\label{lemma-3.1}
Let $S$ be a Hausdorff semitopological semigroup and $I$ be a compact ideal in $S$. Then the Rees-quotient semigroup $S/I$ with the quotient topology is a Hausdorff semitopological semigroup.
\end{lemma}

\begin{theorem}\label{theorem-3.2}
Let $(\mathscr{C}_I,\tau)$ be a Hausdorff locally compact semitopological semigroup, $\mathscr{C}_I=\mathscr{C}(p,q)\sqcup I$ and $I$ is a compact ideal of $\mathscr{C}_I$. Then either $(\mathscr{C}_I,\tau)$ is a compact semitopological semigroup or the ideal $I$ open.
\end{theorem}

\begin{proof}
Suppose that $I$ is not open. By Lemma~\ref{lemma-3.1} the Rees-quotient semigroup $\mathscr{C}_I/I$ with the quotient topology $\tau_{\operatorname{\textsf{q}}}$ is a semitopological semigroup. Let $\pi\colon \mathscr{C}_I\to \mathscr{C}_I/I$ be the natural homomorphism which is a quotient map. It is obvious that the Rees-quotient semigroup $\mathscr{C}_I/I$ is isomorphic to the semigroup $\mathscr{C}^0$ and the image $\pi(I)$ is zero of $\mathscr{C}^0$. Now we shall show that the natural homomorphism $\pi\colon \mathscr{C}_I\to \mathscr{C}_I/I$ is a hereditarily quotient map. Since $\pi(\mathscr{C}(p,q))$ is a discrete subspace of $(\mathscr{C}_I/I,\tau_{\operatorname{\textsf{q}}})$, it is sufficient to show that for every open neighbourhood $U(I)$ of the ideal $I$ in the space $(\mathscr{C}_I,\tau)$ we have that the image  $\pi(U(I))$ is an open neighbourhood of the zero $0$ in the space $(\mathscr{C}_I/I,\tau_{\operatorname{\textsf{q}}})$. Indeed, $\mathscr{C}_I\setminus U(I)$ is a closed-and-open subset of $(\mathscr{C}_I,\tau)$, because the elements of the bicyclic monoid $\mathscr{C}(p,q)$ are isolated point of $(\mathscr{C}_I,\tau)$. Also, since the restriction $\pi|_{\mathscr{C}(p,q)}\colon \mathscr{C}(p,q)\to \pi(\mathscr{C}(p,q))$ of the natural homomorphism $\pi\colon \mathscr{C}_I\to \mathscr{C}_I/I$ is one-to-one,  $\pi(\mathscr{C}_I\setminus U(I))$ is a closed-and-open subset of $(\mathscr{C}_I/I,\tau_{\operatorname{\textsf{q}}})$. So $\pi(U(I))$ is an open neighbourhood of the zero $0$ of the semigroup $(\mathscr{C}_I/I,\tau_{\operatorname{\textsf{q}}})$, and hence the natural homomorphism $\pi\colon \mathscr{C}_I\to \mathscr{C}_I/I$ is a hereditarily quotient map. Since $I$ is a compact ideal of the semitopological semigroup $(\mathscr{C}_I,\tau)$, $\pi^{-1}(y)$ is a compact subset of $(\mathscr{C}_I,\tau)$ for every $y\in \mathscr{C}_I/I$. By Din' N'e T'ong's Theorem (see \cite{Din'-N'e-T'ong-1963} or \cite[3.7.E]{Engelking-1989}), $(\mathscr{C}_I/I,\tau_{\operatorname{\textsf{q}}})$ is a Hausdorff locally compact space. If $I$ is not open then by Theorem~\ref{theorem-2.9} the semitopological semigroup $(\mathscr{C}_I/I,\tau_{\operatorname{\textsf{q}}})$ is topologically isomorphic to $(\mathscr{C}^0,\tau_{\operatorname{\textsf{Ac}}})$ and hence it is compact.
Next we shall prove that the space $(\mathscr{C}_I,\tau)$ is compact. Let $\mathscr{U}=\left\{U_\alpha\colon\alpha\in\mathscr{I}\right\}$ be an arbitrary open cover of $(\mathscr{C}_I,\tau)$. Since $I$ is compact, there exist $U_{\alpha_1},\ldots,U_{\alpha_n}\in\mathscr{U}$ such that $I\subseteq U_{\alpha_1}\cup\cdots\cup U_{\alpha_n}$. Put $U=U_{\alpha_1}\cup\cdots\cup U_{\alpha_n}$. Then $\mathscr{C}_I\setminus U$ is a closed-and-open subset of $(\mathscr{C}_I,\tau)$. Also, since the restriction $\pi|_{\mathscr{C}(p,q)}\colon \mathscr{C}(p,q)\to \pi(\mathscr{C}(p,q))$ of the natural homomorphism $\pi\colon \mathscr{C}_I\to \mathscr{C}_I/I$ is one-to-one, $\pi(\mathscr{C}_I\setminus U(I))$ is a closed-and-open subset of $(\mathscr{C}_I/I,\tau_{\operatorname{\textsf{q}}})$, and hence the image $\pi(\mathscr{C}_I\setminus U(I))$ is finite, because the semigroup $(\mathscr{C}_I/I,\tau_{\operatorname{\textsf{q}}})$ is compact. Thus, the set $\mathscr{C}_I\setminus U$ is finite and hence the space $(\mathscr{C}_I,\tau)$ is compact as well.
\end{proof}

\begin{corollary}\label{corollary-3.3}
If $(\mathscr{C}_I,\tau)$ is a locally compact topological semigroup, $\mathscr{C}_I=\mathscr{C}(p,q)\sqcup I$ and $I$ is a compact ideal of $\mathscr{C}_I$, then the ideal $I$ is open.
\end{corollary}
%%%%%%%%%%%%%%%%%%%%%%%%%%%%%%%%%%%%%%%%%%%%%%%%%%%%%%%%%%%%

\section*{Acknowledgements}

The author acknowledges T. Banakh and A. Ravsky for their comments and suggestions.
%%%%%%%%%%%%%%%%%%%%%%%%%%%%%%%%%%%%%%%%%%%%%%%%%%%%%%%%%%%%
%%%%%%%%%%%%%%%%%%%%%%%%%%%%%%%%%%%%%%%%%%%%%%%%%%%%%%%%%%%%

\end{document}